%% file: eval_MIQP_TAC_full.tex
\documentclass[12pt,draftcls,onecolumn]{IEEEtran}

\usepackage{amsmath}
\usepackage{amssymb}
\usepackage{amsthm}
\usepackage[ansinew]{inputenc}
\pdfoutput=1
\usepackage[T1]{fontenc}
\usepackage{graphicx}
\usepackage{color}
\usepackage{fancyhdr}
\usepackage{float}
\usepackage{algorithm}
\usepackage{algorithmicx}
\usepackage{algpseudocode}
\usepackage{multirow}

\usepackage{acronym}
\acrodef{MPT}{Multi Parametric Toolbox}

\graphicspath{{./pic}}

\usepackage{xcolor,calc}

\usepackage{quad2linpaper}

\begin{document}

\def\figurescale{.7}

\title{Efficient evaluation of mp-MIQP solutions using lifting}

  \author{Alexander Fuchs
\thanks{A. Fuchs is with the Research Centre for Energy Networks, Swiss Federal Institute of Technology (ETH), 8092 Zurich, Switzerland. E-mail: fuchs@fen.ethz.ch .}, 
       Daniel Axehill
\thanks{D. Axehill is with the Division of Automatic Control, Link\"oping University, 58183 Link\"oping, Sweden. E-mail: daniel@isy.liu.se .}, 
           Manfred Morari 
\thanks{M. Morari is with the Automatic Control Laboratory, Swiss Federal Institute of Technology (ETH), 8092 Zurich, Switzerland. E-mail: morari@control.ee.ethz.ch .}
  }
  
\maketitle

\begin{abstract}
  This paper presents an efficient approach for the evaluation of multi-parametric mixed integer quadratic programming (mp-MIQP) solutions, 
  occurring for instance in control problems involving discrete time hybrid systems with quadratic cost.
  Traditionally, the online evaluation requires a sequential comparison of piecewise quadratic value functions.  As the main contribution, we introduce a lifted parameter space in which the piecewise quadratic value functions
  become piecewise affine and can be merged to a single value function defined over a single polyhedral partition without any overlaps.  This enables efficient  point location approaches using a single binary search tree.
  Numerical experiments include a power electronics application and demonstrate an online speedup up to an order of magnitude.  We also show how  the achievable online evaluation time can be traded off against the offline computational time.
\end{abstract}


\section{Introduction}
\subsection{Background and motivation}

The main motivation for this work is control of discrete-time hybrid systems \cite{heemels01:_equiv_hybrid_dynam_model, lunze09:_handb_hybrid_system_contr, bemporad99:_contr_system_integ_logic_dynam_const}
 using Model Predictive Control (MPC) with quadratic cost \cite{maciejowski01:_predic_contr_const,rawlings09:_model_predic_contr}. 
The paper considers the parametric solution to the MPC problem, where the optimal control is computed offline for a set of initial states to reduce the online computational effort \cite{bank82:_non_linear_param_optim,pistikopoulos00:_optim_off_param_optim_tools,bemporad02:_model_predic_contr_based_linear_progr}.
 For hybrid  systems with quadratic cost, 
 the offline computation then requires to solve a multi-parametric Mixed Integer Quadratic Programming (mp-MIQP) problem \cite{borrellimoraribook}. 
Solutions to mp-MIQPs have been proposed based on the solution of Mixed Integer Nonlinear Programming problems \cite{dua02:_multip_progr_approac_mixed_integ}, 
on the enumeration of all switching sequences \cite{borrellimoraribook}, 
on dynamic programming \cite{baotic05:_optim_contr_piecew_affin_system},
and on parametric branch and bound \cite{axehill10_subop_automatica}. 

For efficient evaluation,  a parametric solution needs to be stored in a suitable data structure.
The evaluation approaches in \cite{tondel03:_evaluat,JonGri:2006:IFA_2452,WanJon:2007:IFA_3054}  are focused on solutions with non-overlapping polyhedral partitions, 
    which can be computed for mp-LP, mp-QP or mp-MILP problems \cite{kvasnica04:_multi_param_toolb_mpt}.
   This covers the MPC problem classes of  linear system with linear or quadratic cost and hybrid systems with linear cost. 
Efficient data structures for mp-MIQP problems, occurring for  hybrid systems with quadratic cost, is the main topic of this paper and a more or less unexplored field.
The reason is  that the solution is a pointwise minimizer of intersecting piecewise quadratic functions on overlapping polyhedral partitions.
Therefore, the boundary between optimal regions is not only defined by hyperplanes but also, in general,  by quadratic surfaces. 
The approach in \cite{ChrEtal:ecc:07,FucAxe:2010:IFA_3656} can be used with arbitrary functions defined on overlapping polyhedral partitions,
    but requires  an online sequential search to compare all of the potentially many overlapping value functions defined for the given parameter vector.

\subsection{Contributions}

The evaluation of mp-MIQP solutions defined over \emph{multiple overlapping polyhedral partitions} is traditionally performed in two steps \cite{borrellimoraribook}. 
First, for each partition, the region  containing the parameter vector is determined using a binary search tree \cite{tondel03:_evaluat}. 
Then, the optimal partition is determined using a sequential comparison of the value functions associated with the selected regions. 

The main contribution in this paper is  to  show how  mp-MIQP solutions can be lifted and then merged to an equivalent piecewise affine function defined over a \emph{single polyhedral partition without overlaps}.  
This has a direct impact on the online evaluation time,  which can be significantly reduced using a single search tree eliminating the need for the additional sequential search.

The reason for the  significant reduction of the evaluation time is that the complexity of a search tree evaluation depends logarithmically  on the number of regions in the partition. 
As a result, the evaluation of a single larger search tree for the merged partition requires fewer operations than the evaluation of multiple search trees and the sequential function comparison for the original partitions.

The merging of the lifted partitions can be performed with a standard method for mp-MILP solutions \cite{kvasnica04:_multi_param_toolb_mpt}.
Using the   proposed lifting procedure, the method becomes available for any mp-MIQP solution represented by piecewise quadratic functions over overlapping polyhedral partitions. 
This means that the merging can immediately be combined with any of the above listed solution methods for
mp-MIQP problems,  and also to suboptimal solutions as those computed by the algorithm in~\cite{axehill10_subop_automatica}.

A second contribution of this paper is a new partial merging algorithm. It enables a trade-off between the online and offline complexity of the evaluation of both mp-MILP and mp-MIQP solutions.

\subsection{Paper organization}
The remainder of the paper is organized as follows. 
\sref{probform} defines the evaluation problem of mp-MIQP solutions.
\sref{liftingsols} introduces a lifting procedure and a reformulation of the mp-MIQP solution which enables efficient online evaluation and is the main result of the paper. 
\sref{evalsols}  presents the offline and online algorithm for the evaluation of mp-MIQP solutions and their complexity. 
\sref{partialsols} introduces an algorithm that enables a trade-off between the offline and online complexity for the evaluation of mp-MILP and mp-MIQP solutions. 
\sref{exp} applies the algorithms to three mp-MIQP examples,  showing a reduction of the online evaluation time up to an order of magnitude compared to the traditional evaluation approach.
\sref{conclusion} concludes the paper.

\section{Evaluation problem formulation}
\label{sec:probform}
This section introduces the definitions used to characterize mp-MIQP solutions and states the corresponding evaluation problem.
\subsection{Definitions}
\label{sec:defs}
  \begin{def1}  
  A \emph{polyhedron} $\PP$ in $\Set{R}^n$, is an intersection of a finite number of half-spaces, given in inequality form 
   with $H \in \Set{R}^{m\times n}$ and  $K \in \Set{R}^m$ as
  \begin{equation}
  \Poly{P} = \left\{x \in \Set{R}^n   : \quad Hx\leq K \right\} \quad .
\label{eq:polydef}
  \end{equation}
  \end{def1}  
  \begin{def1}  
  Two polyhedra $\PP_1$ and $\PP_2$ in $\Set{R}^n$ are called \emph{overlapping} when they have common interior points, i.e.
  \begin{equation}
\exists \ x \in \R{n} : \quad H_1x<K_1 \ , \quad H_2x < K_2\qquad .
  \end{equation}
  \end{def1}  

  \begin{def1}  
A \emph{polyhedral set} $\Set{P}$ in $\Set{R}^n$ is a finite collection 
$\Set{P}   = \left\{\PP_{1},\PP_{2},...,\PP_{N}\right\}$
of $|\Set{P}| = N$ polyhedra in $\Set{R}^n$.
The  $i$'th polyhedron  is referred to as $\Set{P}[i] = \PP_i$. 
  \end{def1}  

  \begin{def1}  
A \emph{polyhedral partition} $\PaP$ in $\Set{R}^n$ is a 
polyhedral set in $\Set{R}^n$ whose polyhedra  are not overlapping.
  \end{def1}  

  \begin{def1}  
    The \emph{index set} $\mathcal{I}_\Set{P}(x)$   of a  polyhedral set $\Set{P}$ with $N$ elements in $\Set{R}^n$,  
    and a vector $ x \in \Set{R}^{n}$ 
    is  given by
\begin{equation}
\mathcal{I}_\Set{P}( x) = \{i\in \{1,2,...,N\}  :   x \in \Set{P}[i]\} \quad .
\end{equation}
  \end{def1}

  \begin{def1}  
  \label{def:J}
A \emph{set of quadratic functions} $\Set{J}$ in $\Set{R}^n$ is a finite collection 
$\Set{J}   = \left\{J_{1}(\cdot),J_{2}(\cdot),...,J_{N}(\cdot)\right\}$, 
\begin{equation}
J_i: \Set{R}^n \rightarrow \Set{R},  \qquad J_i( x)   =  x^TA_ix + B_i^Tx + C_i\quad , \label{eq:jpi}
\end{equation}
with $A_i = A_i^T \in \Set{R}^{n\times n}$,  $B_i \in \Set{R}^n$ and  $C_i \in \Set{R}$. 
The  $i$'th quadratic function of $\Set{J}$ is referred to as $\Set{J}[i](\cdot) = J_i(\cdot)$. 
If all $A_i$ are zero matrices, $\Set{J}$ is referred to as \emph{set of affine functions}.
  \end{def1}  

  \begin{def1}  
  A \emph{piecewise quadratic function} $J_{\Set{P},\Set{J}}(\cdot)$ in $\Set{R}^n$ over a polyhedral set $\Set{P}$ in $\Set{R}^n$ with a set of quadratic functions $\Set{J}$ in $\Set{R}^n$
  is a map  
\begin{equation}
J_{\Set{P},\Set{J}}: \Set{R}^n \rightarrow \Set{R}  , \qquad  J_{\Set{P},\Set{J}}(x)  = \min_{i\in\mathcal{I}_\Set{P}(x)}\Set{J}[i](x) \quad .\label{eq:vf}
\end{equation}
If $\Set{J}$ is a set of affine functions, $J_{\Set{P},\Set{J}}(\cdot)$ is referred to as \emph{piecewise affine function}.
  \end{def1}

\subsection{Evaluation problem for mp-MIQP solutions}
\label{sec:plmiqp}
The definitions in \sref{defs} are used to characterize  mp-MIQP solutions: 
\begin{lem1}
The value function of an mp-MIQP problem can be represented as a piecewise quadratic function. 
The value function of an mp-MILP problem can be represented as a piecewise affine function. 
\end{lem1}
\begin{proof}
See, for instance,   \cite{borrellimoraribook}. 
The polyhedral set $\Set{P}$ then consists of multiple overlapping polyhedral partitions, each corresponding to a fixed value of the problem's integer variables.
\end{proof}
The evaluation problem of mp-MIQP solutions requires the solution to the minimization problem in \eref{vf}, 
    which is a point location problem in combination with pairwise comparisons of quadratic functions: 
  \begin{def1}
  \label{def:plmiqp}
   \emph{(PL-MIQP)}:
   Given a 
   piecewise quadratic function $J_{\Set{P},\Set{J}}(\cdot)$ 
   and a vector $x$ in $\R{n}$, determine an index $i^* \in \mathcal{I}_\Set{P}(x)$
   such that 
   \begin{equation}
   \forall i \in \mathcal{I}_\Set{P}(x): \quad \Set{J}[i^*](x)\leq \Set{J}[i]( x) \qquad .
   \end{equation}
If $J_{\Set{P},\Set{J}}(\cdot)$ is a piecewise affine function, the problem is referred to as {\emph{(PL-MILP)}}.
  \end{def1}  
In words, among the polyhedra containing the vector, identify the one with the smallest associated function value. 
\begin{rem1}
The solution  of (PL-MIQP) allows to extract the mp-MIQP optimizer associated with the optimal region index $i^*$ of the given parameter vector $x$. 
It also yields the evaluated value function, 
   \begin{equation}
   \quad J_{\Set{P},\Set{J}}(x) = \Set{J}[i^*](x)\qquad. 
   \end{equation}
\end{rem1}

\section{Lifting mp-MIQP solutions}
\label{sec:liftingsols}
This section presents a lifting procedure for piecewise quadratic functions, which is the main contribution of the paper.
It shows that mp-MIQP value functions have an equivalent piecewise affine representation, thereby enabling efficient evaluation schemes.

\subsection{Motivation of the lifting procedure}

   The \ac{MPT} \cite{kvasnica04:_multi_param_toolb_mpt} provides algorithms to construct data structures for the efficient solution of problem (PL-MILP). 
   For that case, the description of the piecewise affine function $J_{\Set{P},\Set{J}}(\cdot)$ is merged to a function defined over a single polyhedral partition, enabling the construction of a binary search tree  \cite{tondel03:_evaluat} for fast online evaluation. 
These algorithms can however not be directly applied to solve (PL-MIQP), the case with \emph{quadratic} terms, since $J_{\Set{P},\Set{J}}(\cdot)$  
 is then a non-convex piecewise quadratic function defined on regions that,  in general,  are partially defined by quadratic boundary constraints.

It will now be shown how an mp-MIQP solution can be lifted to a piecewise affine formulation in a higher dimensional parameter space. 
The lifted formulation after this transformation is shown to be equivalent to the original formulation.    
Furthermore, the lifted formulation has the structure of an mp-MILP solution, making the standard state-of-the-art methods designed for (PL-MILP) problems available to (PL-MIQP) problems.

\subsection{Definition of the lifting procedure}

\begin{def1}
\label{def:lift}
The \emph{lifting transformation} $L(\cdot)$ of  $\R{n}$,    is defined as 
    \begin{align}
        L: \R{n} \rightarrow  \R{l} \quad,  \qquad l = \frac{n^2 + 3n}{2}  \quad , \label{eq:dimlift}\\
          L(x) = [x_1, x_2,...,x_n, \  x_1^2 , x_1x_2, ..., x_1x_n,\nonumber \\
                 \qquad\qquad  \qquad x_2^2,x_2x_3, ...,  x_2x_n, \   ..., x_n^2 ]^T \quad .
    \label{eq:Ldef}
    \end{align}
\end{def1}

\begin{def1}
\label{def:liftp}
Given a polyhedral set $\Set{P}$ in $\R{n}$, the \emph{lifted polyhedral set} \textnormal{$\Set{P}_\text{l} = \frak{L_P}(\Set{P})$} in 
$\R{l}$,  $l = (n^2 + 3n)/2$, is defined as 
\textnormal{
\begin{equation}
\forall i \in   \{1,...,  |\Set{P}|\} : \quad 
\Set{P}_{\text{l}}[i]  = \left\{y \in \Set{R}^l   : [H_i, \mathbf{0}]y\leq K_i \right\} 
\end{equation}
}
where $(H_i,K_i)$ are the matrices defining 
\begin{equation}
\Set{P}[i] = \left\{x \in \Set{R}^n   : H_ix\leq K_i \right\}
\end{equation}
and $\mathbf{0}$ denotes the zero matrix of appropriate dimensions. 
\end{def1}

\begin{rem1}
Through the lifting, the polyhedra $\Set{P}[i]$ are extended along the dimensions  of the lifted space corresponding to the bilinear terms in \eref{Ldef}.
The lifting of the polyhedra does not add constraints or change the structure of their projection on the original $n$ dimensions. 
\end{rem1}

\begin{def1}
\label{def:liftj}
Given a set of quadratic functions $\Set{J}$ in $\R{n}$, the \emph{lifted set of affine functions} 
\textnormal{
$\Set{J}_\text{l} = \frak{L_J}(\Set{J})$ 
}
    in 
$\R{l}$,  $l = (n^2 + 3n)/2$, is defined as 

\textnormal{
\begin{equation}
\forall i \in   \{1,...,  |\Set{J}|\} : \quad
 \Set{J}_\text{l}[i]:  \R{l}\rightarrow \R{} \ ,  \quad 
  \Set{J}_\text{l}[i](y)   =  D_i^Ty + E_i \qquad,  
\end{equation}
}
where \begin{align}
D_{i} & = [B_{i,1}, B_{i,2}, ..., B_{i,n}, A_{i,11}, 2A_{i,12},  ... , 2A_{i,1n}, \nonumber\\
    & \qquad \qquad    A_{i,22}, 2A_{i,23},... ,2A_{i,2n}, ... , A_{i,nn}]^T \quad ,\label{eq:al}\\
     E_{i} & = C_i \quad , \label{eq:bl}
\end{align}
are the rearranged  parameters of the quadratic functions $\Set{J}[i](x)   =  x^TA_ix + B_i^Tx + C_i$. 
\end{def1}

\subsection{Properties of lifted mp-MIQP solutions}
\label{sec:proplift}
The following results show the equivalence of piecewise quadratic functions and the corresponding lifted piecewise affine functions.
\begin{lem1}
\label{lem:pp}
Given a polyhedral set $\Set{P}$ in $\R{n}$, the lifting transformation $L(\cdot)$,  and the lifted polyhedral set 
\textnormal{
$\Set{P}_\text{l} = \frak{L_P}(\Set{P})$}, it holds that
\textnormal{
\begin{equation} 
\forall x \in \R{n}:\quad \mathcal{I}_\Set{P}(x) = \mathcal{I}_{\Set{P}_\text{l}}(L(x)) \qquad .
\label{eq:IIl}
\end{equation}
}
\end{lem1}

\begin{proof}
\begin{align}
\forall x \in  \R{n} , \ \forall i \in & \{ 1, ...,|\Set{P}|\} : \nonumber\\
 i \in\mathcal{I}_\Set{P}(x)
& \leftrightarrow 
   x\in \Set{P}[i] \leftrightarrow
   H_ix\leq K_i \nonumber\\
& \leftrightarrow
  [H_i, \mathbf{0}]
  [x^T, x_1^2 , x_1x_2, ..., x_n^2]^T \leq K_i \nonumber\\
& \leftrightarrow
  L(x)\in \Set{P}_\text{l}[i] \nonumber\\
& \leftrightarrow 
  i \in\mathcal{I}_{\Set{P}_\text{l}}(L(x)) \nonumber
\end{align}
\end{proof}

\begin{lem1}
\label{lem:jj}
Given a set of quadratic functions $\Set{J}$ in $\R{n}$, the lifting transformation $L(\cdot)$, and the lifted set of affine functions 
\textnormal{
$\Set{J}_\text{l} = \frak{L_{J}}(\Set{J})$}, it holds that
\textnormal{
\begin{equation} 
\forall x \in \R{n}, \ \forall i \in \{1,...,|\Set{J}|\} :\quad 
\Set{J}[i](x) = \Set{J}_\text{l}[i](L(x)) \qquad .
\label{eq:JJl}
\end{equation}
}
\end{lem1}

\begin{proof}
\begin{align}
\forall x \in \R{n} , \ \forall i & \in \{1,...,|\Set{P}|\} : \nonumber\\
 \Set{J}[i](x) & =  x^TA_ix + B_i^Tx + C_i  \nonumber\\
               & = B_{i,1}x_1 + ... + B_{i,n}x_n \nonumber\\
               & \quad + A_{i,11}x_1^2 + 2A_{i,12}x_1x_2 + ... + 2A_{i,1n}x_1x_n \nonumber\\
               & \quad + A_{i,22}x_2^2+ 2A_{i,23}x_2x_3 + ... + 2A_{i,2n}x_2x_n \nonumber\\
               & \quad + ... + A_{i,nn}x_n^2 + C_i \nonumber \\
               & =  D_i^T L(x) + E_i \nonumber\\
               & = \Set{J}_\text{l}[i](L(x)) \nonumber
\end{align}
\end{proof}

\begin{theorem}
Given a piecewise quadratic function $J_{\Set{P},\Set{J}}(\cdot)$ in $\R{n}$, 
      a lifting transformation $L(\cdot)$ and the lifted sets 
\textnormal{
     $\Set{P}_\text{l} = \frak{L_{P}}(\Set{P})$}
     and 
\textnormal{
     $\Set{J}_\text{l} = \frak{L_{J}}(\Set{J})$}, 
     the piecewise affine function 
\textnormal{
     $J_{\Set{P}_\text{l},\Set{J}_\text{l}}(\cdot)$ } 
     satisfies 

\textnormal{
\begin{equation} 
\forall x \in \R{n} :\quad 
J_{\Set{P},\Set{J}}(x) = J_{\Set{P}_\text{l},\Set{J}_\text{l}}(L(x)) \qquad .
\label{eq:pwJJl}
\end{equation}}
\end{theorem}

\begin{proof}
\begin{align}
\forall x \in   \R{n}  : 
J_{\Set{P},\Set{J}}(x) & \stackrel{\hphantom{\text{\scriptsize{(Lemma \ref{lem:pp})}}}}{ =}  &&\min_{i\in\mathcal{I}_\Set{P}(x)}\Set{J}[i](x)   \nonumber\\
                      &  \stackrel{\text{\scriptsize{(Lemma \ref{lem:pp})}}}{ =}  && \min_{i\in\mathcal{I}_{\Set{P}_\text{l}}(L(x))}\Set{J}[i](x)  \nonumber\\
                       & \stackrel{\text{\scriptsize{(Lemma \ref{lem:jj})}}}{ =}  && \min_{i\in\mathcal{I}_{\Set{P}_\text{l}}(L(x))}\Set{J}_\text{l}[i](L(x)) \nonumber\\
                       &  \stackrel{\hphantom{\text{\scriptsize{(Lemma \ref{lem:pp})}}}}{ =} && J_{\Set{P}_\text{l},\Set{J}_\text{l}}(L(x))  \qquad .\nonumber
\end{align}
\end{proof}

\begin{rem1}
\label{rem:lift}
The construction of the lifted piecewise affine function 
\textnormal{
$J_{\Set{P}_\text{l},\Set{J}_\text{l}}(\cdot)$} is computationally inexpensive.
Both the polyhedral set 
\textnormal{
$\Set{P}_\text{l}$} and the set of affine functions 
\textnormal{
$\Set{J}_\text{l}$} require only a rearrangement of the data representing the original piecewise quadratic function $J_{\Set{P},\Set{J}}(\cdot)$.
\end{rem1}
\begin{rem1}
Since Theorem 1 states that value functions of mp-MIQP problems can be represented as equivalent piecewise affine functions  
\textnormal{
$J_{\Set{P}_\text{l},\Set{J}_\text{l}}(\cdot)$}, algorithms for efficient evaluation of mp-MILP solutions,  such as  \cite{kvasnica04:_multi_param_toolb_mpt,tondel03:_evaluat},  can now be directly applied. 
\end{rem1}

\section{Evaluation of mp-MIQP solutions}
\label{sec:evalsols}
The proposed offline procedure to prepare the evaluation of mp-MIQP solutions, 
    stated as problem (PL-MIQP) in Definition \ref{def:plmiqp}, 
    is summarized as \aref{prep}.
All lines of \aref{prep} use existing algorithms 
available in the Multi-Parametric Toolbox \cite{kvasnica04:_multi_param_toolb_mpt} in MATLAB, 
    except for the lifting operation in line two.
The algorithms and their complexity are discussed in the following four subsections. 
Subsection \ref{sec:onlineeval} then discusses the online evaluation and its complexity. 

\begin{algorithm}
\begin{algorithmic}[1]
    \caption{\textsc{PrepareEvaluation}($\Set{P}$, $\Set{J}$)}
\label{alg:prep}
\Require set of polyhedral partitions $\Set{P}$, 
set of quadratic functions $\Set{J}$
    \Ensure binary search tree $\mathcal{T}$
    \State  $ (\Set{P}_\text{r},\Set{J}_\text{r}) \gets   \textsc{Reduce}(\Set{P},\Set{J}) $ 
    \State  $ (\Set{P}_\text{l},\Set{J}_\text{l}) \gets   (\mathfrak{L_P}(\Set{P}_\text{r}),\mathfrak{L_J}(\Set{J}_\text{r}))$ 
    \State  $ (\PaP,\Set{J}_\text{m}) \gets   \textsc{Merge}(\Set{P}_\text{l},\Set{J}_\text{l})$ 
    \State  $ \mathcal{T} \gets   \textsc{Tree}(\PaP)$ 
\end{algorithmic}
\end{algorithm}

\subsection{Overlap reduction}
  A piecewise quadratic function $J_{\Set{P},\Set{J}}(\cdot)$ in $\R{n}$ may contain regions $\Set{P}[i]$ whose associated quadratic function value $\Set{J}[i](x)$ is \emph{never} 
  minimizing the expression in \eref{vf} 
  for any vector $x \in \Set{P}[i]$.  
  These regions can be identified and removed 
  using the algorithm in \cite{besselmann10:_const_optim_contr}, denoted by $\textsc{Reduce}$ in the first line of \aref{prep}. 
For  $N$ initial polyhedra in $\Set{P}$, $\textsc{Reduce}$  solves up to $N^2$ indefinite quadratic programs with $n$ variables to identify the reduced polyhedral set $\Set{P}_\text{r}$ and the associated quadratic functions $\Set{J}_\text{r}$.
The complexity can be reduced using several heuristics.
Since line three of \aref{prep} also removes redundant regions, the application of $\textsc{Reduce}$ could be omitted but serves as a preprocessing step to improve performance. 

\subsection{Lifting}
The second line of \aref{prep} applies the lifting operation  to the 
$|\Set{P}_\text{r}|$ 
polyhedra and quadratic functions of $(\Set{P}_\text{r},\Set{J}_\text{r})$, as defined in Definitions  \ref{def:liftp} and \ref{def:liftj}. 
  As pointed out in Remark \ref{rem:lift}, this is a formal rearrangement of the internal data representation and requires no additional computations.

\subsection{Merging}
The third line of \aref{prep} removes the region overlaps of the piecewise affine function $J_{\Set{P}_\text{l},\Set{J}_\text{l}}(\cdot)$ in $\R{l}$, defined by 
the lifted polyhedral set and the set of affine functions $(\Set{P}_\text{l},\Set{J}_\text{l})$. 
The algorithm $\textsc{Merge}$ is shown as pseudocode in \aref{remlin}, based on the implementation in \cite{kvasnica04:_multi_param_toolb_mpt}. 
It provides a single polyhedral partition $\PaP$ and a set of affine functions $\Set{J}_\text{m}$,  defining the equivalent piecewise affine function $J_{\PaP,\Set{J}_\text{m}}(\cdot)$ in $\R{l}$. 
Since the regions of $\PaP$ do not overlap, a function value comparison is no longer necessary and the original (PL-MIQP) reduces to a pure point location problem in the partition $\PaP$. 

\begin{algorithm}
\begin{algorithmic}[1]
    \caption{\textsc{Merge}($\Set{P}$, $\Set{J}$), \cite{kvasnica04:_multi_param_toolb_mpt}}
\label{alg:remlin}
\Require polyhedral set $\Set{P}$, 
    set of affine functions $\Set{J} $
    \Ensure polyhedral partition $\PaP$,  \newline 
    set of affine functions $\Set{J}_\text{m}$
\State $\PaP \gets \emptyset, \Set{J}_\text{m} \gets \emptyset$ 
\For{$i \in \{1,...,|\Set{P}|\}$ }
    \State $\Set{Q} \gets \emptyset$
    \For{$j \in \{1,...,|\Set{P}|\}, j \neq i$}
        \State $\Poly{Q} \gets \{x\in (\Set{P}[i]\cap \Set{P}[j]):  \Set{J}[i](x) \geq \Set{J}[j](x)\}$
        \If{$\Poly{Q}\neq \emptyset$}
        \State $\Set{Q} \gets [\Set{Q},\Poly{Q}]$
        \EndIf
    \EndFor
    \State $\mathcal{D} \gets \textsc{RegionDiff}(\Set{P}[i], \Set{Q})$
    \If{$\mathcal{D}\neq \emptyset$}
        \State $\PaP \gets [\PaP,\mathcal{D}]$
        \For{$k \in \{1,...,|\mathcal{D}|\}$ }
            \State $\Set{J}_\text{m} \gets [\Set{J}_\text{m},\Set{J}[i](\cdot)]$
        \EndFor
    \EndIf
\EndFor
\end{algorithmic}
\end{algorithm}

Essentially, \aref{remlin} 
loops through all polyhedra  $\Set{P}[i]$ and collects a polyhedral set $\Set{Q}$, covering the subset of $\Set{P}[i]$ where the function $\Set{J}[i](\cdot)$ \emph{is not the minimizer} 
of all overlapping functions in $\Set{J}$.
Consequently the set difference $\mathcal{D} = \Set{P}[i]\setminus \Set{Q}$ is the portion of $\Set{P}[i]$ where the function $\Set{J}[i](\cdot)$ \emph{is the minimizer} of all overlapping functions in $\Set{J}$.

The construction of  $\Set{Q}$ inside the two for-loops requires less than $|\Set{P}|^2$
feasibility checks of polyhedra. 
The central and most expensive part  of   {\aref{remlin}} is the set difference operation \textsc{RegionDiff}, called $|\Set{P}|$ times in line ten.
A pessimistic upper bound on the number of LPs solved by \textsc{RegionDiff} is given in \cite{Mato}. The bound is exponential in the dimension of $\Set{Q}$, with the total number of constraints of all polyhedra as base. 

The lifting of mp-MIQP solutions to $\R{l}$, as in Definition \ref{def:lift},  practically squares the problem dimension compared to mp-MILP solutions 
in the original parameter space $\R{n}$.
The exponential bound on the number of LPs solved by \textsc{RegionDiff} suggests that merging lifted mp-MIQP solutions is much more expensive than merging mp-MILP solutions 
with a similar polyhedral structure. 
However, numerical experiments  indicate that the total complexity of \aref{remlin} with affine or lifted quadratic functions is similar. 
The reason is that in the case with lifted quadratic functions, only very few constraints of $\Set{Q}$, originating from the $|\Set{Q}|$ function differences in line five of \aref{remlin}, actually spread along the bilinear dimensions. 
In other words, because of the special structure of the lifting procedure, only a small amount of additional complexity is introduced to \textsc{Merge} when applied to a lifted problem (PL-MIQP) compared to a problem (PL-MILP) with the same underlying polyhedral set.

\subsection{Search tree construction}
After the merging, the solution of a problem (PL-MIQP) reduces to a point location in a single polyhedral partition $\PaP$.
An efficient solution is   the construction of a binary search tree,  denoted by  
$\textsc{Tree}$
    in line four of {{\aref{prep}}}. 

An algorithm to construct a  binary tree using  the polyhedra's hyperplanes  as decision criteria is given in \cite{tondel03:_evaluat}.  The method uses heuristics to obtain a balanced tree. 
A central part of the algorithm is the  preprocessing step that determines the relative position of every polyhedron and each of the $n_\text{h}$  hyperplanes of the partition, solving up to $2n_\text{h}|\PaP|$ LPs.
Constructing a tree that is guaranteed to have minimum depth might require the solution of an MILP with up to $2|\PaP|$ binary variables for each node of the tree \cite{FucJon:2010:IFA_3625}.
The method can also be generalized to trees with more than two children \cite{partree}, which are particularly suitable for an implementation with multiple processors. 

Neither $n_\text{h}$ nor $|\PaP|$ are directly increased through the additional dimensions from the lifting.
In other words, only little additional complexity is introduced to \textsc{Tree} when solving  a lifted problem (PL-MIQP) instead of a problem (PL-MILP) with a similar underlying polyhedral set.

\subsection{Online evaluation}
 \label{sec:onlineeval}
After the preparation with \aref{prep}, the solution of a problem (PL-MIQP) with a polyhedral set $\Set{P}$, a set of quadratic functions $\Set{J}$ and a vector $x$ in  $\R{n}$ reduces to the evaluation of the binary tree $\mathcal{T}$.  
The evaluation is a sequence of vector multiplications \cite{tondel03:_evaluat} that needs to be applied to the lifted vector $y = L(x)$, defined in \eref{dimlift}. 
It is  denoted by 
\begin{equation}
i^* \gets \textsc{EvaluateTree}(\mathcal{T},y)
\label{eq:evaltree}
\end{equation}
and returns the index of the optimal region $\Set{P}[i^*]$. 
For mp-MIQP solutions, each region has an associated control law that can now be extracted. 
A balanced binary tree can execute point location queries in $\log_2(|\PaP|)$ tree node decisions, where $\PaP$ is denoting the polyhedral partition after the merging \cite{tondel03:_evaluat}. 

   It is of interest how the online evaluation complexity 
of mp-MIQP solutions compares with and without the preparation through the lifting and merging procedure in \aref{prep}. 
While \sref{exp} shows a numerical assessment with concrete examples, a basic comparison is obtained as follows. Consider a piecewise quadratic function $J_{\Set{P},\Set{J}}(\cdot)$ in $\R{n}$
defined over $n_\text{part}$ partitions,   with the same number of $m$ polyhedra in each partition.
Without \aref{prep}, each of the $n_\text{part}$ partitions is evaluated with a separate  search tree \cite{borrellimoraribook}.
        The total number of online operations for the tree evaluations then is 
\begin{equation}
N_{\text{ops},\text{no merging}} = n_\text{part}\cdot K_1\cdot\log_2(m)\qquad, 
\label{eq:nopsnomerge}
\end{equation}
where $K_1$, the number of arithmetic operations per tree node decision, grows linearly with the problem dimension $n$. Additionally,  $n_\text{part}$ operations are required to find the optimal partition. 
In comparison, using \aref{prep}, the fully merged partition $\PaP$ is evaluated with a single tree, requiring 
\begin{equation}
N_{\text{ops},\text{merging}} = K_2  \cdot \log_2(|\PaP|)\qquad
\label{eq:nopsmerge}
\end{equation}
operations. 
The factor $K_2$ is slightly larger than $K_1$,  depending on the ratio of tree decisions involving the lifted dimensions,  $n+1$ to $l$. 
It follows that a reduction of the online complexity through the lifting and merging is given whenever
\begin{align}
  N_{\text{ops},\text{merging}}\ & <\  N_{\text{ops},\text{no merging}} \qquad,\\ 
\leftrightarrow  \quad  |\PaP| \ & <\   m^{\frac{K_1\cdot n_\text{part}}{K_2}} \ \approx\ m^{n_\text{part}} \qquad. 
\end{align}
For mp-MIQP solutions to practical problem instances, one often obtains $|\PaP| \ll m^{n_\text{part}}$, leading to a significant improvement of the evaluation time when using the lifting and merging procedure. This is also confirmed by the examples in \sref{exp}.

\section{Partial merging of mp-MIQP solutions}
\label{sec:partialsols}

This section presents a modification of lines three and four in \aref{prep} that allows to choose a trade-off between offline and online complexity. It can be used for both mp-MIQP and mp-MILP solutions. 

\subsection{Pairwise partition merging}
\label{sec:scalemerge}
If the offline preparation using \aref{prep} can not be completed within the available offline computational time, it is still possible to improve the online evaluation time by partially merging the solution's partitions. 
To define the partial merging algorithm, the partition structure of mp-MIQP solutions is characterized using the following additional definitions.
In words, each polyhedron of the polyhedral set $\Set{P}$ is assigned to one of the partitions of the mp-MIQP solution.

  \begin{def1}  
    A \emph{partition index} $\Set{I}$ for $n_\text{\emph{part}}$ partitions, associated with a polyhedral set of $N$ elements,  
    is  given by
\begin{equation}
\Set{I} = \{s_1,s_2,...,s_N \}, \qquad s_i \in \{1,2, ...,n_\text{\emph{part}}\}\quad.
\end{equation}
  \end{def1}  
  \begin{def1}  
    The \emph{index set of the $k$'th polyhedral partition} for a  partition index $\Set{I}$ with $N$ elements   
    is  given by
\begin{equation}
\mathcal{I}_\Set{I}(k) = \{i\in \{1,2,...,N\}  :    \Set{I}[i] = k \} \quad .
\end{equation}
  \end{def1}  
  \begin{def1}  
The elements of a polyhedral set $\Set{P}$ and  a set of functions $\Set{J}$ corresponding to an  index $\mathcal{I} = \{i_1,i_2,...\}$ are denoted by 
$\Set{P}[\mathcal{I}] = \{\Set{P}[i_1],\Set{P}[i_2],...\}$ and $\Set{J}[\mathcal{I}] = \{\Set{J}[i_1](\cdot),\Set{J}[i_2](\cdot),...\}$.
  \end{def1}  

The partial merging algorithm \textsc{MergePairwise} is given in \aref{pair} and replaces line three in \aref{prep}. It runs $n_\text{m}$ iterations, each of which merges pairs of polyhedral partitions, using 
\textsc{Merge}, as defined in \aref{remlin}. 
The associated affine function set $\Set{J}$ can originate directly from an mp-MILP solution or from a lifted mp-MIQP solution. 
Line nine then assigns the number of the merged partition, $k$, to all elements of the corresponding new partition index.  
The complexity of each function call of \textsc{Merge} depends on the number of polyhedral constraints of its argument and grows with the number of iterations $n_\text{m}$, 
which can therefore be used to select the offline complexity of \aref{pair}.

The selection of the pairing through the partition index in line six of  \aref{pair} is arbitrary.  
It can be adjusted to consider generalized polyhedral subsets, as long as they cover the full polyhedral set $\Set{P}$.
A greedy heuristic to obtain a small polyhedral set $\Set{P}_\text{m}$ is to execute \textsc{MergePairwise} repeatedly with different permutations of the partition index $\Set{I}$, keeping the one that yields the smallest $|\Set{P}_\text{m}|$.

\begin{algorithm}
\begin{algorithmic}[1]
    \caption{\textsc{MergePairwise}($\Set{P}$, $\Set{J}$, $\Set{I}$, $n_\text{m}$)}
\label{alg:pair}
\Require polyhedral set $\Set{P}$, 
    set of affine functions $\Set{J}$, 
    partition index $\Set{I}$,
    number of merging iterations $n_\text{m}$
    \Ensure 
polyhedral set ${\Set{P}_\text{m}}$, 
set of affine functions ${\Set{J}_\text{m}}$, 
reduced partition index ${\Set{I}_\text{m}}$
    \State $(\Set{P}_\text{m}, \Set{J}_\text{m}, \Set{I}_\text{m})\gets (\Set{P}, \Set{J}, \Set{I})$
    \While{$n_\text{m}>0  $}
        \State $k \gets 0$ 
        \While{${k}< \max ({\Set{I}_\text{m}})/2$}
            \State $k\gets k+1  , \ \Set{I}_\text{k}\gets \emptyset$
            \State ${\mathcal{I}} \gets  [ \mathcal{I}_\Set{I}(2{k}-1) , \  \mathcal{I}_\Set{I}(2{k})]$
            \State  $ ({\PaP}_k,{\Set{J}}_k) \gets   \textsc{Merge}({\Set{P}_\text{m}[{\mathcal{I}}]},{\Set{J}_\text{m}[{\mathcal{I}}]})$ 
            \For{$i\in \{1,..., |\PaP_k|\}$}
                \State $\Set{I}_k[i] \gets k  $
            \EndFor
        \EndWhile
        \State ${\Set{P}_\text{m}} \gets [{\PaP}_1,..., {\PaP}_k ]$
        \State ${\Set{J}_\text{m}} \gets [{\Set{J}}_1,..., {\Set{J}}_k ]$
        \State ${\Set{I}_\text{m}} \gets [\Set{I}_1,..., \Set{I}_k ]$
        \State $n_\text{m}\gets n_\text{m}-1$
    \EndWhile
\end{algorithmic}
\end{algorithm}

After the partial merging of the partitions, the tree construction in line four  of \aref{prep} must be performed for each one of the remaining partitions.
The corresponding algorithm \textsc{MultiTree} is defined in \aref{multtree}. 
\begin{algorithm}
\begin{algorithmic}[1]
    \caption{\textsc{MultiTree}($\Set{P}$, $\Set{I}$)}
\label{alg:multtree}
\Require polyhedral set $\Set{P}$, 
    partition index $\Set{I}$, 
    \Ensure 
set of search trees ${\Set{T}}$, 
    \For{$i\in \{1,..., \max(\Set{I})\}$}
        \State $\mathcal{T}_i \gets \textsc{Tree}(\Set{P}[\mathcal{I}_\Set{I}(i)])$
    \EndFor
    \State $\Set{T} = [\mathcal{T}_1, \mathcal{T}_2,..., \mathcal{T}_{\max(\Set{I})}] $
\end{algorithmic}
\end{algorithm}

\subsection{Online evaluation}
     For partially merged mp-MIQP solutions, the optimal region is analogously to the traditional look-up methods determined using a standard two step procedure \cite{borrellimoraribook}, which is denoted by  \textsc{EvaluateMultiTrees} and shown as pseudocode in \aref{evalmult}. 
     After first evaluating the binary trees in line five  of \aref{evalmult},  the corresponding value function values are compared to determine the index of the optimal region. 
If the evaluation of the $n_\text{t} = |\Set{T}|$ trees  requires $N_{\text{ops},i}$ scalar operations for the $i$'th tree, 
   a total of 
     \begin{equation}
N_{\text{ops}} =  (l+1)n_\text{t} + \sum_{i=1}^{n_\text{t}} N_{\text{ops},i}
     \label{eq:nops}
     \end{equation}
operations are required for the online evaluation of the mp-MIQP solution. 
This includes additions, multiplications and comparisons.  
The first term in \eref{nops} accounts for the value function comparisons in line six of \aref{evalmult} in the  
$l$-dimensional lifted space.
The number $N_{\text{ops},i}$ is the maximum number of operations for a single execution of \textsc{EvaluateTree}.
It depends logarithmically on the size of the polyhedral partition and corresponds to the term $K_1\cdot\log_2(m)$ of the basic complexity estimate in  \eref{nopsnomerge}.

\begin{algorithm}
\begin{algorithmic}[1]
    \caption{\textsc{EvaluateMultiTrees}($\Set{T}$, $\Set{J}$, $\Set{I}$, $x$)}
\label{alg:evalmult}
\Require set of search trees $\Set{T} $, set of functions  $\Set{J}$,
    partition index $\Set{I}$, 
    vector  $x$
\Ensure optimal index $(i^*, j^*)$
    \State  $ y \gets L(x) $ 
    \State  $ J^* \gets \infty $ 
    \For{$i \in\{1,..., \max(\Set{I})\} $}
        \State{$ \Set{J}_i \gets \Set{J}[\mathcal{I}_\Set{I}[i]]$  }
        \State $ j \gets \textsc{EvaluateTree}(\Set{T}[i],y) $
        \If{$\Set{J}_i[j](x) < J^*$}
            \State $J^* \gets \Set{J}_i[j](x)$
            \State $(i^*,j^*) \gets (i,j)$
        \EndIf
    \EndFor
\end{algorithmic}
\end{algorithm}

\section{Numerical experiments}
\label{sec:exp}
In this section, the proposed approach to  evaluate mp-MIQP solutions is applied to  three example cases.
First, a simple artificial problem illustrates  \aref{prep}. 
Second, control of a simple PWA system shows the potential reduction of online complexity. 
Finally, the algorithm is applied to a recent approach for controlling DC-DC converters, showing the trade-off between online and offline complexity 
in an industrially relevant application.

\subsection{Illustrative 1D example}
\label{sec:epx1D}
This section illustrates the steps of  \aref{prep} using an artificial 1D-example of two overlapping polyhedra 
$\Set{P}=\{|x|\leq 2, |x|\leq 3\}$
with corresponding intersecting quadratic functions $\Set{J}=\{x^2 + 1, 2x^2\}$ , shown in \fref{s1}. 
After the lifting operation in line two of \aref{prep}, the polyhedra  $\Set{P}_\text{l}$
    and the functions $\Set{J}_\text{l}$ are defined over the space $\{x,x^2\}$, but still intersect (\fref{s3}).
The \textsc{Merge} operation in line three of \aref{prep} then provides a single partition with no overlaps 
and the corresponding piecewise affine function $\Set{J}_\text{m}$ (\fref{s5}).
It is now possible to build a binary search tree for the merged partition.
\begin{figure}
    \center
    \def\svgwidth{\figurescale\columnwidth} 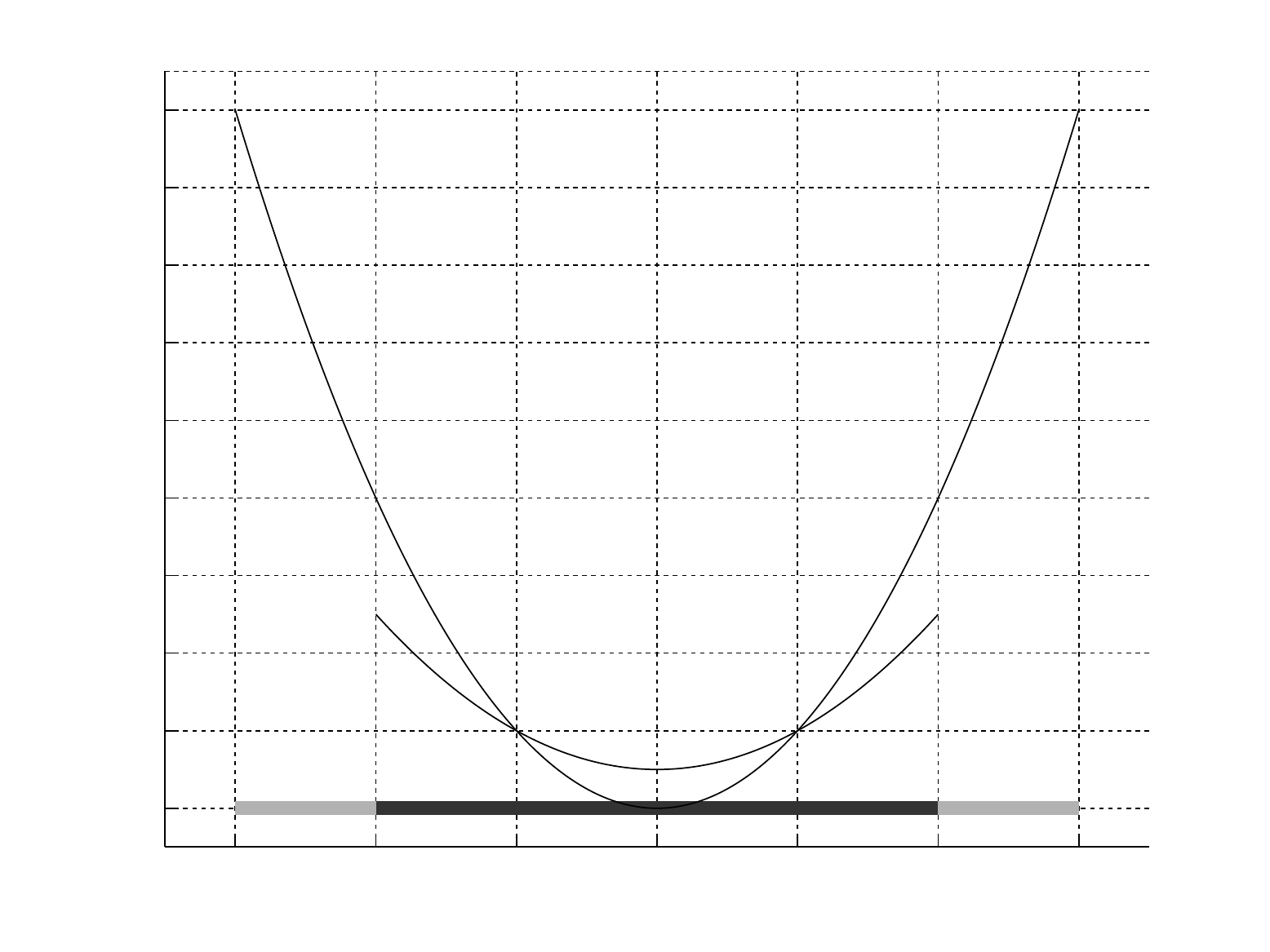
    \caption{Illustrative example: Two overlapping polyhedra $\Set{P}$ (bold lines) with corresponding
     intersecting quadratic functions $\Set{J}$ (thin curves). }
    \label{fig:s1}
\end{figure}
\begin{figure}
    \center
    \def\svgwidth{\figurescale\columnwidth} 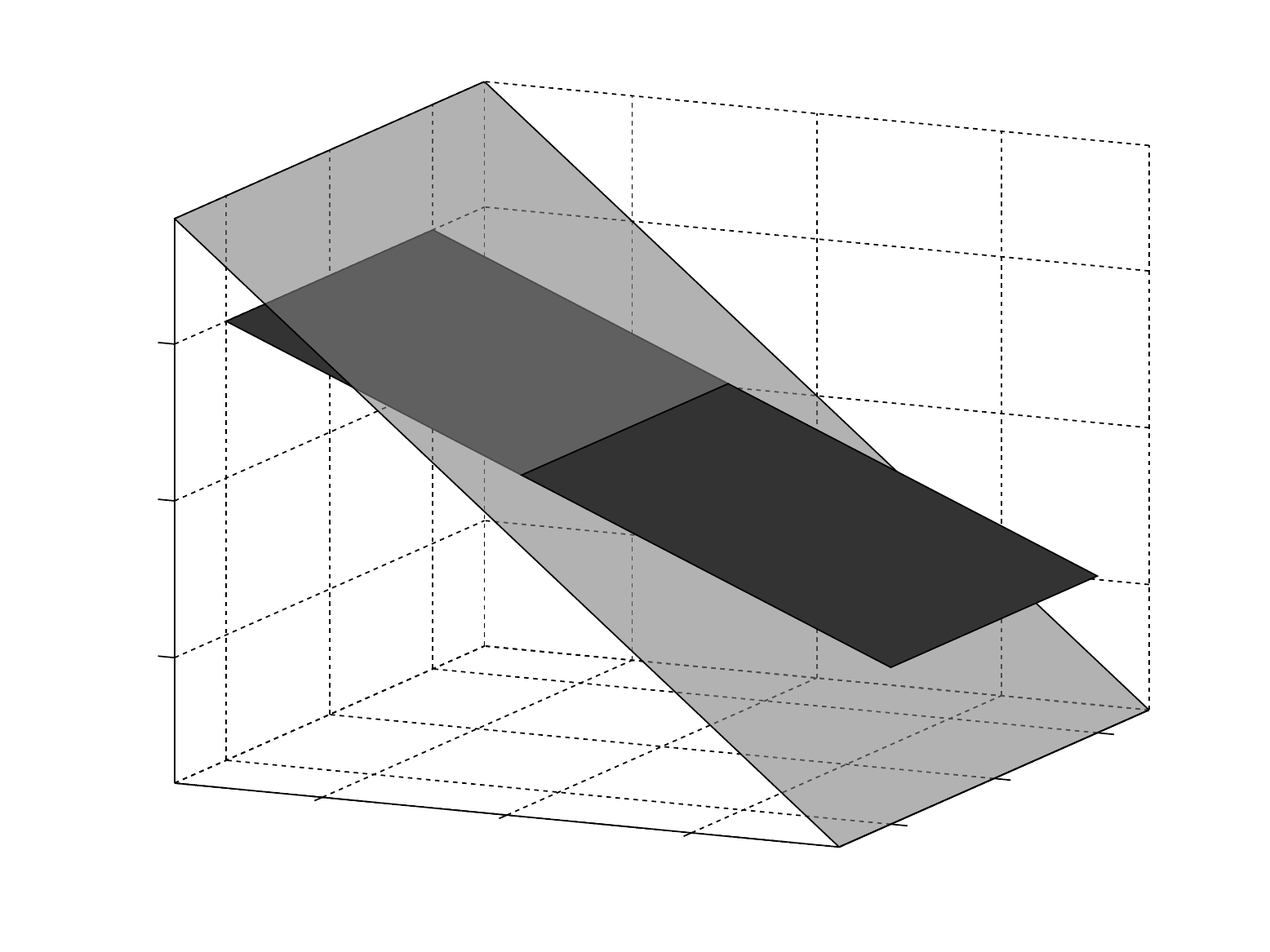
    \caption{Illustrative example: The lifted functions $\Set{J}_\text{l}$ are affine.}
    \label{fig:s3}
\end{figure}
\begin{figure}
    \center
    \def\svgwidth{\figurescale\columnwidth} 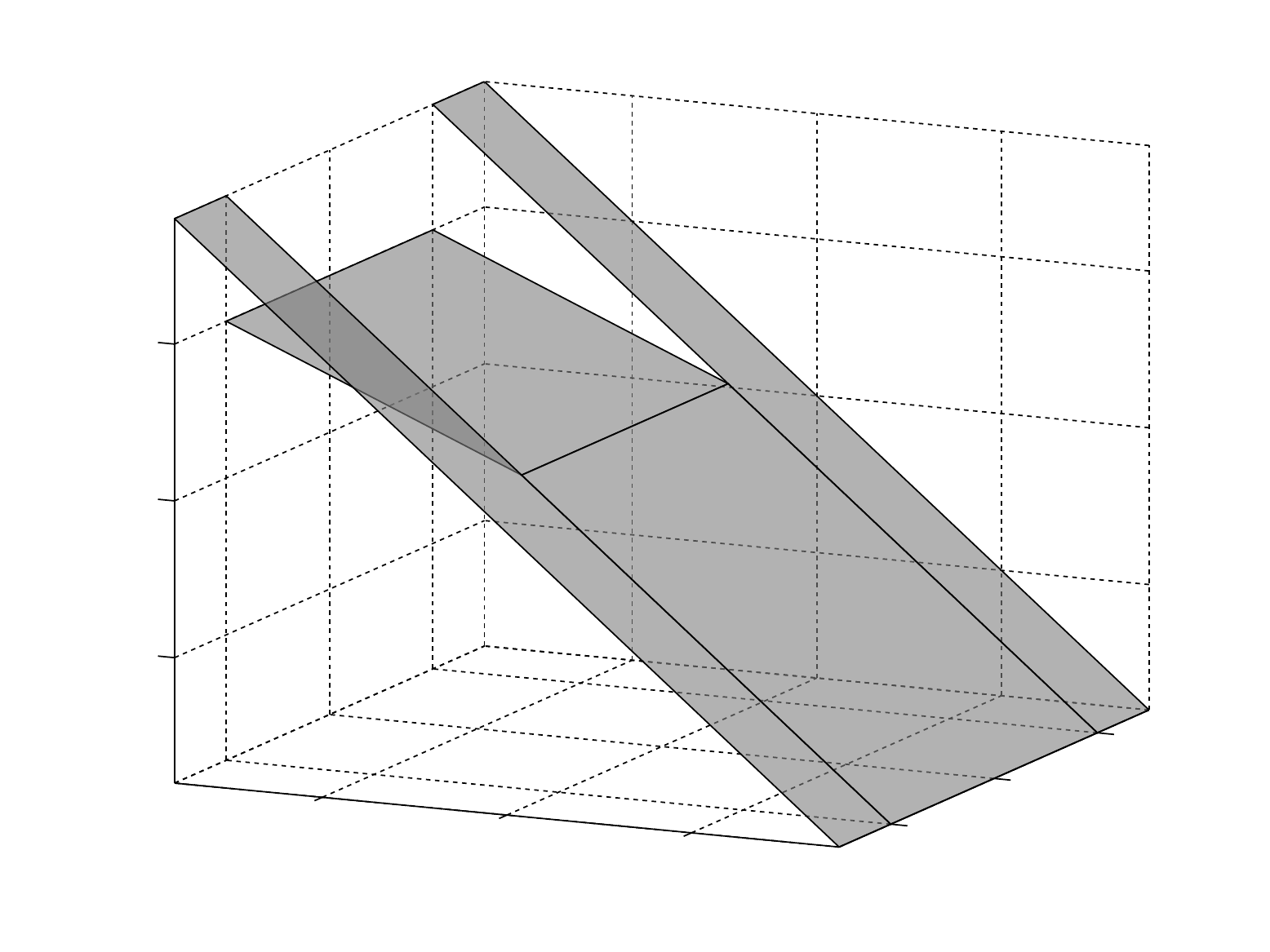
    \caption{Illustrative example: Functions $\Set{J}_\text{m}$ after merging.}
    \label{fig:s5}
\end{figure}

\subsection{2D PWA Example}

The PWA system given by equation (44) in  \cite{bemporad99:_contr_system_integ_logic_dynam_const} has two dynamic states, one input,  two different dynamics and box constraints on states and input.
The system is controlled using a finite horizon formulation with the penalty matrices 
\begin{equation}
Q = \left(\begin{array}{cc}1 &0 \\0 & 1  \end{array}\right)  \qquad
R = 1\quad .
\end{equation}

\fref{2dpwaNops} shows the total number of online operations $N_{ \text{ops}}$, defined in \eref{nops}, as a function of the prediction horizon $N$. 
Lifting and merging the solution reduces $N_{\text{ops}}$ by more than an order of magnitude. The factor increases with the prediction horizon.
In particular, the evaluation of the merged mp-MIQP solution with a prediction horizon $N=6$ is still faster than the traditional approach with $N=1$.
As shown in \tref{sol2Dpwa}, the merging operation increases the total number of polyhedra $n_\text{p}$, thereby also increasing the size of the resulting search tree.
  The number  floating point numbers stored in the search trees, $n_\text{store}$, determines the memory footprint of the control law. For $N=6$, the merging operation causes  $n_\text{store}$ to increase by a factor of six, compared to the traditional approach without merging.

\label{tab:sol2Dpwa}
\begin{figure}
    \center
    \def\svgwidth{\figurescale\columnwidth} 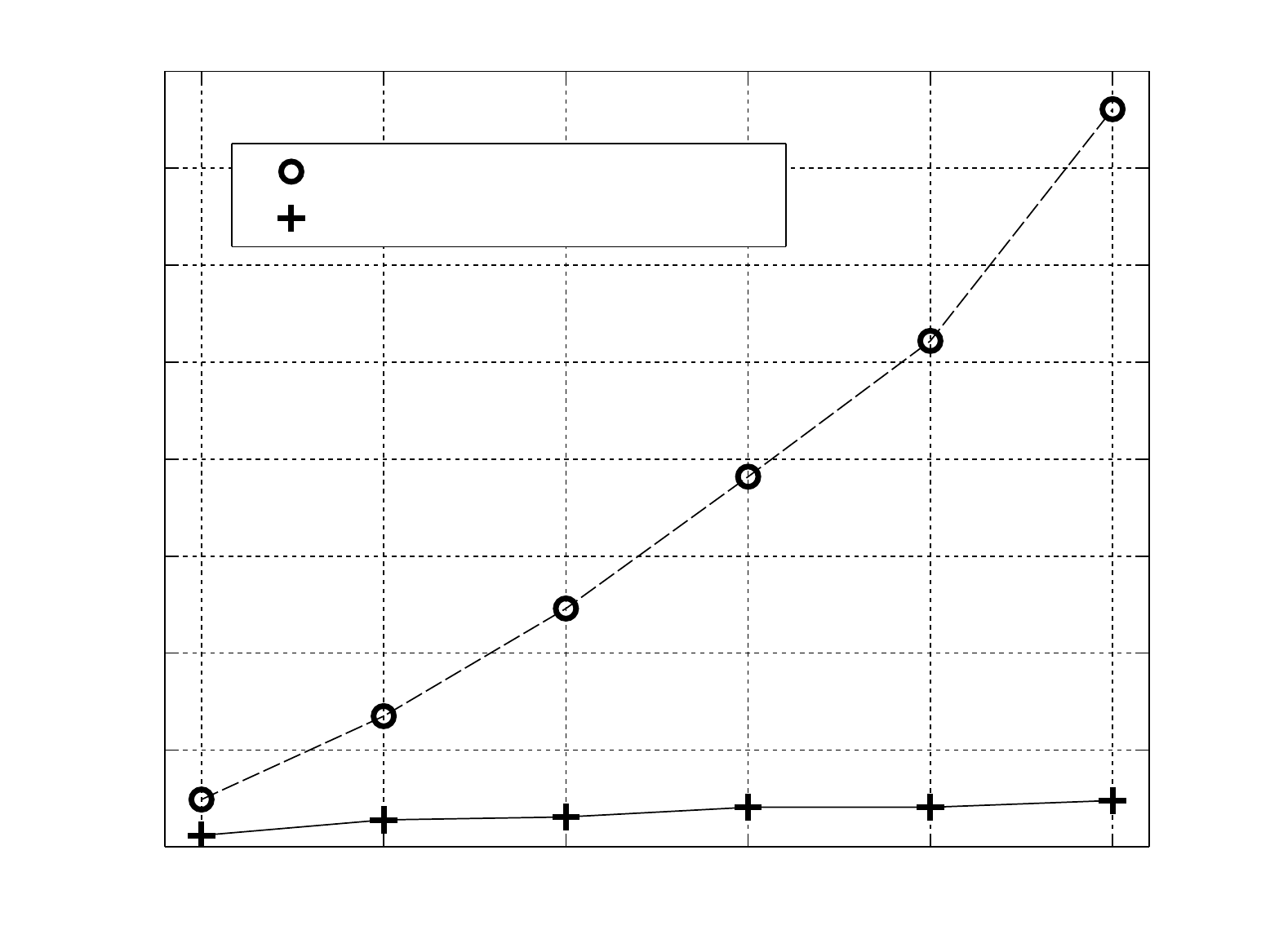
    \caption{2D PWA system: Number of online operations $N_{\text{ops}}$ with different predictions horizons $N$, using traditional approach (dashed) 
    and \aref{prep} (solid) .}
    \label{fig:2dpwaNops}
\end{figure}

\begin{table}[h]\centering
\begin{tabular}{ccccccc}
$N$ & 1 & 2 & 3& 4 & 5 & 6 \\ \hline\hline
\multicolumn{7}{c}{without merging }\\
 $n_\text{t}$& 4 & 8 & 14& 22 & 32 & 48\\ 
 $n_\text{p}$& 12 & 47 & 100& 168 & 221 & 322\\ 
 $n_\text{store}$ & 24 & 174 & 360 & 666 & 813 & 1248\\\hline
\multicolumn{7}{c}{complete merging $\quad$ ($n_\text{t}=1$) }\\
 $n_\text{p}$ & 12 & 87 & 208& 587 & 650 & 1560\\
 $n_\text{store}$ & 33 & 366  & 780 & 2718 & 2487 & 7395\\
 &&&&&&
\end{tabular}
\caption{2D PWA system: number of partitions \textnormal{$n_\text{t}$}, number of polyhedra \textnormal{$n_\text{p}$}, and number of floating point values stored in the tree(s) \textnormal{$n_\text{store}$}, for different prediction horizons $N$.}
\label{tab:sol2Dpwa}
\end{table}

\subsection{5D DC-DC converter example  }

A recent approach to the control of  DC-DC converters \cite{dcdc} uses 
a mixed logical dynamical system formulation to compute an explicit receding horizon control policy. The equivalent formulation as PWA system has five states, one input, three different dynamics and 
box constraints on states and inputs. In \cite{dcdc}, the system has been controlled using a 1-norm stage cost $||Qx ||_1$. 
The control approach is applied to the same system formulation, only changing the cost functions to the 2-norm stage cost $x^TQx + u^TRu $ with
\begin{equation}
Q = \text{diag}([4 \  0.1 \ 0 \ 0 \ 0 ])\quad ,  \qquad R = 0.001 \quad .
\end{equation}
In the lifted space, which has the dimension $l=20$, it is now possible to merge the overlapping partitions for an efficient implementation of the resulting control policy.

\fref{dcdcNops} shows the total number of online operations $N_{ \text{ops}}$ defined in \eref{nops}, 
    as a function of the number of merging iterations $n_\text{m}$.
For the shown cases, merging the solution in the lifted space reduces $N_{\text{ops}}$ up to a factor of seven compared to the traditional approach without merging.
The factor increases with the prediction horizon $N$ and  the number of merging iterations $n_\text{m}$.
In particular, the evaluation of the completely merged mp-MIQP solution with $N=4$ is faster than the traditional approach to evaluate a solution with $N=1$.

The offline effort for a prediction horizon $N=4$ and different number of merging iterations $n_\text{m}$ is summarized in  \tref{solDCDC}. For the preparation of the evaluation using a single search tree, 
a binary tree was constructed for a partition of 13821 20-dimensional polyhedra.
The time of the merging operation itself remained relatively small (about 15 minutes), compared to the time of the tree construction (about 36 hours), 
    both using a simple MATLAB implementation on a single core machine. 
This also confirms  that the lifting  does not render the merging problem intractable due to the increased number of dimensions.

\begin{figure}
    \center
    \def\svgwidth{\figurescale\columnwidth} 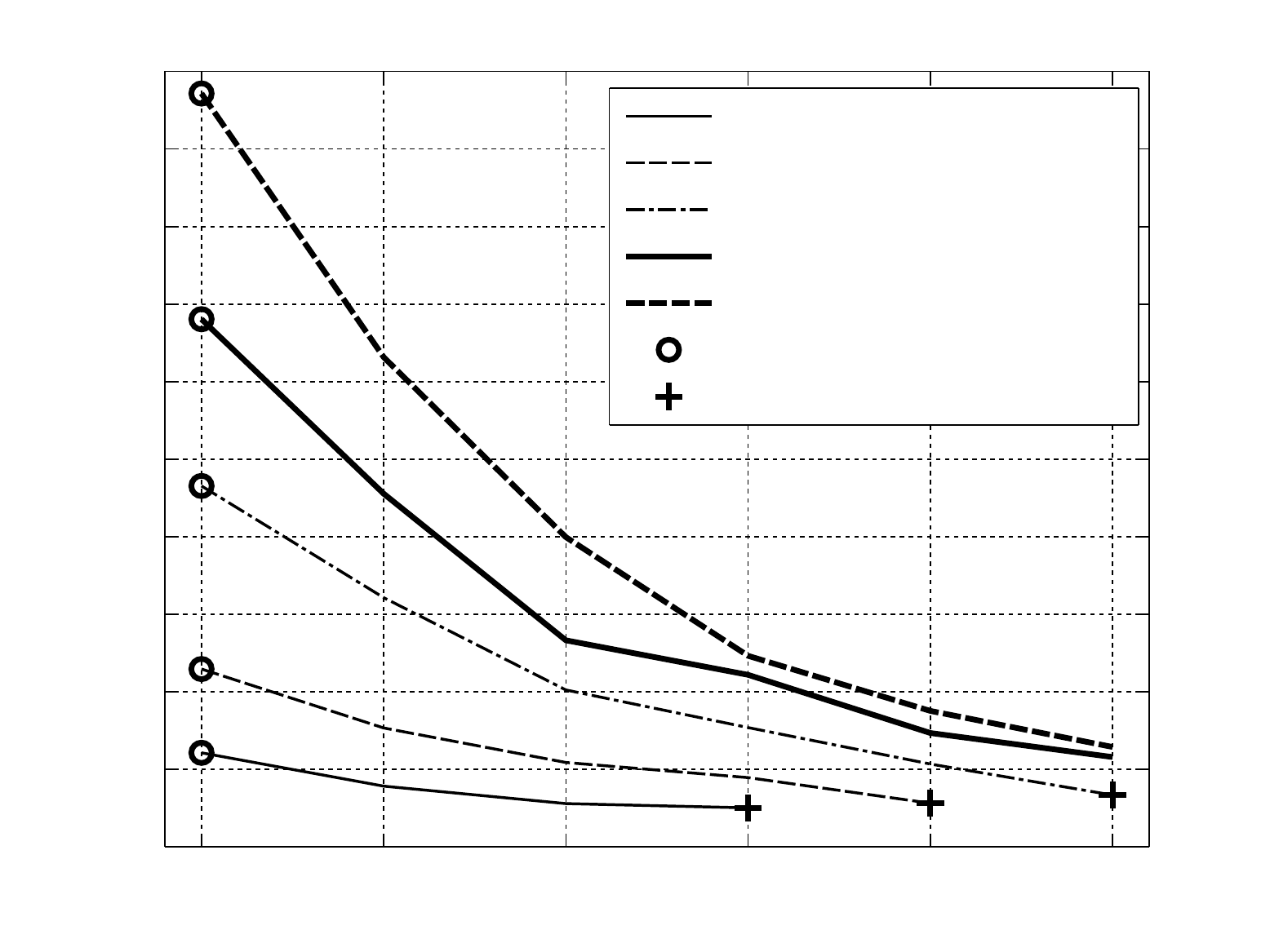
    \caption{DC-DC converter problem: Number of online operations $N_{\text{ops}}$ with $n_\text{m}$ merging recursions of \aref{pair} for different prediction horizons $N$.}
    \label{fig:dcdcNops}
\end{figure}

\begin{table}[h]\centering
\begin{tabular}{ccccccc}
$n_\text{m}$ & 0 & 1 & 2& 3 & 4 & 5 \\ \hline\hline
 $n_\text{t}$& 31 & 16 & 8& 4 & 2 & 1\\ 
 $n_\text{p}$& 105 & 104 & 111& 267 & 1170& 13821\\ 
 $n_\text{store}$ & 546 & 702 & 954 & 2766 & 12690 & 138402 \\
 $t_\text{merge}$ & 0 & 1 & 2& 10& 51& 906\\
 $t_\text{tree}$ & 8& 17& 32& 99& 810& 130000\\
 &&&&&&
\end{tabular}
\caption{DC-DC converter problem with prediction horizon $N=4$: number of partitions \textnormal{$n_\text{t}$}, number of polyhedra \textnormal{$n_\text{p}$}, number of floating point values stored in the tree(s) \textnormal{$n_\text{store}$},  
offline time to merge partitions \textnormal{$t_\text{merge}$ [seconds]}, offline time to build search trees \textnormal{$t_\text{tree}$ [seconds]}, 
 for different merging recursions \textnormal{$n_\text{m}$}.}
\label{tab:solDCDC}
\end{table}

\section{Conclusion}
\label{sec:conclusion}
  The evaluation of mp-MIQP solutions requires a comparison of potentially many overlapping piecewise quadratic value functions defined on polyhedral sets.
  In this paper it is shown how the quadratic functions and the associated polyhedra can be lifted to a higher dimensional parameter space.
 It is shown that  mp-MIQP  solutions  in this space have a representation as polyhedral piecewise affine function without overlaps. 
 For the online evaluation, this enables the use of efficient data structures known from {mp-MILP} problems, including binary search trees.
 Furthermore, an algorithm is presented that enables a trade-off between online and off\-line computational complexity both for mp-MILP and mp-MIQP problems.
  The numerical examples include a  power electronics control problem  
  of practical relevance. An online speedup up to an order of magnitude is achieved.

\bibliographystyle{IEEEtran}
\bibliography{references}

\end{document}

%% file: s1.pdf_tex

\begingroup
  \makeatletter
  \providecommand\color[2][]{%
    \errmessage{(Inkscape) Color is used for the text in Inkscape, but the package 'color.sty' is not loaded}
    \renewcommand\color[2][]{}%
  }
  \providecommand\transparent[1]{%
    \errmessage{(Inkscape) Transparency is used (non-zero) for the text in Inkscape, but the package 'transparent.sty' is not loaded}
    \renewcommand\transparent[1]{}%
  }
  \providecommand\rotatebox[2]{#2}
  \ifx\svgwidth\undefined
    \setlength{\unitlength}{446.4pt}
  \else
    \setlength{\unitlength}{\svgwidth}
  \fi
  \global\let\svgwidth\undefined
  \makeatother
  \begin{picture}(1,0.74731183)%
    \put(0,0){\includegraphics[width=\unitlength]{s1}}%
    \put(0.51660215,0.0195){\color[rgb]{0,0,0}\makebox(0,0)[b]{\smash{$x$}}}%
    \put(0.03515591,0.38404659){\color[rgb]{0,0,0}\rotatebox{90}{\makebox(0,0)[b]{\smash{${\mathbb{J}}$}}}}%
    \put(0.18535663,0.04887634){\color[rgb]{0,0,0}\makebox(0,0)[b]{\smash{-3}}}%
    \put(0.29607168,0.04887634){\color[rgb]{0,0,0}\makebox(0,0)[b]{\smash{-2}}}%
    \put(0.40678495,0.04887634){\color[rgb]{0,0,0}\makebox(0,0)[b]{\smash{-1}}}%
    \put(0.5175,0.04887634){\color[rgb]{0,0,0}\makebox(0,0)[b]{\smash{0}}}%
    \put(0.62821505,0.04887634){\color[rgb]{0,0,0}\makebox(0,0)[b]{\smash{1}}}%
    \put(0.73892832,0.04887634){\color[rgb]{0,0,0}\makebox(0,0)[b]{\smash{2}}}%
    \put(0.84964337,0.04887634){\color[rgb]{0,0,0}\makebox(0,0)[b]{\smash{3}}}%
    \put(0.1145,0.10459319){\color[rgb]{0,0,0}\makebox(0,0)[rb]{\smash{0}}}%
    \put(0.1145,0.16549821){\color[rgb]{0,0,0}\makebox(0,0)[rb]{\smash{2}}}%
    \put(0.1145,0.22640502){\color[rgb]{0,0,0}\makebox(0,0)[rb]{\smash{4}}}%
    \put(0.1145,0.28731004){\color[rgb]{0,0,0}\makebox(0,0)[rb]{\smash{6}}}%
    \put(0.1145,0.34821685){\color[rgb]{0,0,0}\makebox(0,0)[rb]{\smash{8}}}%
    \put(0.1145,0.40912186){\color[rgb]{0,0,0}\makebox(0,0)[rb]{\smash{10}}}%
    \put(0.1145,0.47002867){\color[rgb]{0,0,0}\makebox(0,0)[rb]{\smash{12}}}%
    \put(0.1145,0.53093369){\color[rgb]{0,0,0}\makebox(0,0)[rb]{\smash{14}}}%
    \put(0.1145,0.5918405){\color[rgb]{0,0,0}\makebox(0,0)[rb]{\smash{16}}}%
    \put(0.1145,0.65274552){\color[rgb]{0,0,0}\makebox(0,0)[rb]{\smash{18}}}%
  \end{picture}%
\endgroup

%% file: s3.pdf_tex

\begingroup
  \makeatletter
  \providecommand\color[2][]{%
    \errmessage{(Inkscape) Color is used for the text in Inkscape, but the package 'color.sty' is not loaded}
    \renewcommand\color[2][]{}%
  }
  \providecommand\transparent[1]{%
    \errmessage{(Inkscape) Transparency is used (non-zero) for the text in Inkscape, but the package 'transparent.sty' is not loaded}
    \renewcommand\transparent[1]{}%
  }
  \providecommand\rotatebox[2]{#2}
  \ifx\svgwidth\undefined
    \setlength{\unitlength}{448pt}
  \else
    \setlength{\unitlength}{\svgwidth}
  \fi
  \global\let\svgwidth\undefined
  \makeatother
  \begin{picture}(1,0.75)%
    \put(0,0){\includegraphics[width=\unitlength]{s3}}%
    \put(0.80772321,0.02712321){\color[rgb]{0,0,0}\makebox(0,0)[lb]{\smash{$y_1= x$}}}%
    \put(0.44182321,0.01284464){\color[rgb]{0,0,0}\makebox(0,0)[rb]{\smash{$y_2= x^2$}}}%
    \put(0.0589625,0.35108929){\color[rgb]{0,0,0}\rotatebox{90}{\makebox(0,0)[b]{\smash{$\mathbb{J}_\text{l}$}}}}%
    \put(0.72818571,0.06839286){\color[rgb]{0,0,0}\makebox(0,0)[b]{\smash{-2}}}%
    \put(0.80935536,0.1041875){\color[rgb]{0,0,0}\makebox(0,0)[b]{\smash{0}}}%
    \put(0.890525,0.13998214){\color[rgb]{0,0,0}\makebox(0,0)[b]{\smash{2}}}%
    \put(0.53326964,0.05875893){\color[rgb]{0,0,0}\makebox(0,0)[b]{\smash{-5}}}%
    \put(0.3882125,0.07266786){\color[rgb]{0,0,0}\makebox(0,0)[b]{\smash{0}}}%
    \put(0.24315536,0.08657679){\color[rgb]{0,0,0}\makebox(0,0)[b]{\smash{5}}}%
    \put(0.113175,0.22209107){\color[rgb]{0,0,0}\makebox(0,0)[rb]{\smash{-10}}}%
    \put(0.113175,0.34491964){\color[rgb]{0,0,0}\makebox(0,0)[rb]{\smash{0}}}%
    \put(0.113175,0.46775){\color[rgb]{0,0,0}\makebox(0,0)[rb]{\smash{10}}}%
  \end{picture}%
\endgroup

%% file: s5.pdf_tex

\begingroup
  \makeatletter
  \providecommand\color[2][]{%
    \errmessage{(Inkscape) Color is used for the text in Inkscape, but the package 'color.sty' is not loaded}
    \renewcommand\color[2][]{}%
  }
  \providecommand\transparent[1]{%
    \errmessage{(Inkscape) Transparency is used (non-zero) for the text in Inkscape, but the package 'transparent.sty' is not loaded}
    \renewcommand\transparent[1]{}%
  }
  \providecommand\rotatebox[2]{#2}
  \ifx\svgwidth\undefined
    \setlength{\unitlength}{448pt}
  \else
    \setlength{\unitlength}{\svgwidth}
  \fi
  \global\let\svgwidth\undefined
  \makeatother
  \begin{picture}(1,0.75)%
    \put(0,0){\includegraphics[width=\unitlength]{s5}}%
    \put(0.80772321,0.02712321){\color[rgb]{0,0,0}\makebox(0,0)[lb]{\smash{$y_1= x$}}}%
    \put(0.44182321,0.01284464){\color[rgb]{0,0,0}\makebox(0,0)[rb]{\smash{$y_2= x^2$}}}%
    \put(0.0589625,0.35108929){\color[rgb]{0,0,0}\rotatebox{90}{\makebox(0,0)[b]{\smash{$\mathbb{J}_\text{m}$}}}}%
    \put(0.72818571,0.06839286){\color[rgb]{0,0,0}\makebox(0,0)[b]{\smash{-2}}}%
    \put(0.80935536,0.1041875){\color[rgb]{0,0,0}\makebox(0,0)[b]{\smash{0}}}%
    \put(0.890525,0.13998214){\color[rgb]{0,0,0}\makebox(0,0)[b]{\smash{2}}}%
    \put(0.53326964,0.05875893){\color[rgb]{0,0,0}\makebox(0,0)[b]{\smash{-5}}}%
    \put(0.3882125,0.07266786){\color[rgb]{0,0,0}\makebox(0,0)[b]{\smash{0}}}%
    \put(0.24315536,0.08657679){\color[rgb]{0,0,0}\makebox(0,0)[b]{\smash{5}}}%
    \put(0.113175,0.22209107){\color[rgb]{0,0,0}\makebox(0,0)[rb]{\smash{-10}}}%
    \put(0.113175,0.34491964){\color[rgb]{0,0,0}\makebox(0,0)[rb]{\smash{0}}}%
    \put(0.113175,0.46775){\color[rgb]{0,0,0}\makebox(0,0)[rb]{\smash{10}}}%
  \end{picture}%
\endgroup

%% file: sincos.pdf_tex

\begingroup
  \makeatletter
  \providecommand\color[2][]{%
    \errmessage{(Inkscape) Color is used for the text in Inkscape, but the package 'color.sty' is not loaded}
    \renewcommand\color[2][]{}%
  }
  \providecommand\transparent[1]{%
    \errmessage{(Inkscape) Transparency is used (non-zero) for the text in Inkscape, but the package 'transparent.sty' is not loaded}
    \renewcommand\transparent[1]{}%
  }
  \providecommand\rotatebox[2]{#2}
  \ifx\svgwidth\undefined
    \setlength{\unitlength}{448pt}
  \else
    \setlength{\unitlength}{\svgwidth}
  \fi
  \global\let\svgwidth\undefined
  \makeatother
  \begin{picture}(1,0.75)%
    \put(0,0){\includegraphics[width=\unitlength]{sincos}}%
    \put(0.51660714,0.01996071){\color[rgb]{0,0,0}\makebox(0,0)[b]{\smash{$N$}}}%
    \put(0,0.38544464){\color[rgb]{0,0,0}\rotatebox{90}{\makebox(0,0)[b]{\smash{$N_{\text{ops}}$}}}}%
    \put(0.15870357,0.04920357){\color[rgb]{0,0,0}\makebox(0,0)[b]{\smash{1}}}%
    \put(0.30222143,0.04920357){\color[rgb]{0,0,0}\makebox(0,0)[b]{\smash{2}}}%
    \put(0.44574107,0.04920357){\color[rgb]{0,0,0}\makebox(0,0)[b]{\smash{3}}}%
    \put(0.58925893,0.04920357){\color[rgb]{0,0,0}\makebox(0,0)[b]{\smash{4}}}%
    \put(0.73277857,0.04920357){\color[rgb]{0,0,0}\makebox(0,0)[b]{\smash{5}}}%
    \put(0.87629643,0.04920357){\color[rgb]{0,0,0}\makebox(0,0)[b]{\smash{6}}}%
    \put(0.1145,0.07446429){\color[rgb]{0,0,0}\makebox(0,0)[rb]{\smash{0}}}%
    \put(0.1145,0.15086964){\color[rgb]{0,0,0}\makebox(0,0)[rb]{\smash{100}}}%
    \put(0.1145,0.22727679){\color[rgb]{0,0,0}\makebox(0,0)[rb]{\smash{200}}}%
    \put(0.1145,0.30368214){\color[rgb]{0,0,0}\makebox(0,0)[rb]{\smash{300}}}%
    \put(0.1145,0.38008929){\color[rgb]{0,0,0}\makebox(0,0)[rb]{\smash{400}}}%
    \put(0.1145,0.45649464){\color[rgb]{0,0,0}\makebox(0,0)[rb]{\smash{500}}}%
    \put(0.1145,0.53290179){\color[rgb]{0,0,0}\makebox(0,0)[rb]{\smash{600}}}%
    \put(0.1145,0.60930714){\color[rgb]{0,0,0}\makebox(0,0)[rb]{\smash{700}}}%
    \put(0.1145,0.68571429){\color[rgb]{0,0,0}\makebox(0,0)[rb]{\smash{800}}}%
    \put(0.27080179,0.60510714){\color[rgb]{0,0,0}\makebox(0,0)[lb]{\smash{\tiny{no merging (traditional)}}}}%
    \put(0.27080179,0.5679125){\color[rgb]{0,0,0}\makebox(0,0)[lb]{\smash{\tiny{completely merged}}}}%
  \end{picture}%
\endgroup

%% file: dcdc.pdf_tex

\begingroup
  \makeatletter
  \providecommand\color[2][]{%
    \errmessage{(Inkscape) Color is used for the text in Inkscape, but the package 'color.sty' is not loaded}
    \renewcommand\color[2][]{}%
  }
  \providecommand\transparent[1]{%
    \errmessage{(Inkscape) Transparency is used (non-zero) for the text in Inkscape, but the package 'transparent.sty' is not loaded}
    \renewcommand\transparent[1]{}%
  }
  \providecommand\rotatebox[2]{#2}
  \ifx\svgwidth\undefined
    \setlength{\unitlength}{448pt}
  \else
    \setlength{\unitlength}{\svgwidth}
  \fi
  \global\let\svgwidth\undefined
  \makeatother
  \begin{picture}(1,0.75)%
    \put(0,0){\includegraphics[width=\unitlength]{dcdc}}%
    \put(0.51660714,0.01996071){\color[rgb]{0,0,0}\makebox(0,0)[b]{\smash{$n_{\text{m}}$}}}%
    \put(0,0.38544464){\color[rgb]{0,0,0}\rotatebox{90}{\makebox(0,0)[b]{\smash{$N_{\text{ops}}$}}}}%
    \put(0.15870357,0.04920357){\color[rgb]{0,0,0}\makebox(0,0)[b]{\smash{0}}}%
    \put(0.30222143,0.04920357){\color[rgb]{0,0,0}\makebox(0,0)[b]{\smash{1}}}%
    \put(0.44574107,0.04920357){\color[rgb]{0,0,0}\makebox(0,0)[b]{\smash{2}}}%
    \put(0.58925893,0.04920357){\color[rgb]{0,0,0}\makebox(0,0)[b]{\smash{3}}}%
    \put(0.73277857,0.04920357){\color[rgb]{0,0,0}\makebox(0,0)[b]{\smash{4}}}%
    \put(0.87629643,0.04920357){\color[rgb]{0,0,0}\makebox(0,0)[b]{\smash{5}}}%
    \put(0.1145,0.07446429){\color[rgb]{0,0,0}\makebox(0,0)[rb]{\smash{0}}}%
    \put(0.1145,0.13558929){\color[rgb]{0,0,0}\makebox(0,0)[rb]{\smash{200}}}%
    \put(0.1145,0.19671429){\color[rgb]{0,0,0}\makebox(0,0)[rb]{\smash{400}}}%
    \put(0.1145,0.25783929){\color[rgb]{0,0,0}\makebox(0,0)[rb]{\smash{600}}}%
    \put(0.1145,0.31896429){\color[rgb]{0,0,0}\makebox(0,0)[rb]{\smash{800}}}%
    \put(0.1145,0.38008929){\color[rgb]{0,0,0}\makebox(0,0)[rb]{\smash{1000}}}%
    \put(0.1145,0.44121429){\color[rgb]{0,0,0}\makebox(0,0)[rb]{\smash{1200}}}%
    \put(0.1145,0.50233929){\color[rgb]{0,0,0}\makebox(0,0)[rb]{\smash{1400}}}%
    \put(0.1145,0.56346429){\color[rgb]{0,0,0}\makebox(0,0)[rb]{\smash{1600}}}%
    \put(0.1145,0.62458929){\color[rgb]{0,0,0}\makebox(0,0)[rb]{\smash{1800}}}%
    \put(0.1145,0.68571429){\color[rgb]{0,0,0}\makebox(0,0)[rb]{\smash{2000}}}%
    \put(0.56849286,0.64853929){\color[rgb]{0,0,0}\makebox(0,0)[lb]{\smash{$N = 2$}}}%
    \put(0.56849286,0.61264286){\color[rgb]{0,0,0}\makebox(0,0)[lb]{\smash{$N = 3$}}}%
    \put(0.56849286,0.57495357){\color[rgb]{0,0,0}\makebox(0,0)[lb]{\smash{$N = 4$}}}%
    \put(0.56849286,0.5372625){\color[rgb]{0,0,0}\makebox(0,0)[lb]{\smash{$N = 5$}}}%
    \put(0.56849286,0.50136786){\color[rgb]{0,0,0}\makebox(0,0)[lb]{\smash{$N = 6$}}}%
    \put(0.56849286,0.46367857){\color[rgb]{0,0,0}\makebox(0,0)[lb]{\smash{\tiny{no merging (traditional)}}}}%
    \put(0.56849286,0.42778214){\color[rgb]{0,0,0}\makebox(0,0)[lb]{\smash{\tiny{completely merged}}}}%
  \end{picture}%
\endgroup

%% file: eval_MIQP_TAC_full.bbl
\begin{thebibliography}{10}
\providecommand{\url}[1]{#1}
\csname url@samestyle\endcsname
\providecommand{\newblock}{\relax}
\providecommand{\bibinfo}[2]{#2}
\providecommand{\BIBentrySTDinterwordspacing}{\spaceskip=0pt\relax}
\providecommand{\BIBentryALTinterwordstretchfactor}{4}
\providecommand{\BIBentryALTinterwordspacing}{\spaceskip=\fontdimen2\font plus
\BIBentryALTinterwordstretchfactor\fontdimen3\font minus
  \fontdimen4\font\relax}
\providecommand{\BIBforeignlanguage}[2]{{%
\expandafter\ifx\csname l@#1\endcsname\relax
\typeout{** WARNING: IEEEtran.bst: No hyphenation pattern has been}%
\typeout{** loaded for the language `#1'. Using the pattern for}%
\typeout{** the default language instead.}%
\else
\language=\csname l@#1\endcsname
\fi
#2}}
\providecommand{\BIBdecl}{\relax}
\BIBdecl

\bibitem{heemels01:_equiv_hybrid_dynam_model}
W.~Heemels, B.~de~Schutter, and A.~Bemporad, ``{Equivalence of Hybrid Dynamical
  Models},'' \emph{Automatica}, vol.~37, no.~7, pp. 1085 -- 1091, 2001.

\bibitem{lunze09:_handb_hybrid_system_contr}
J.~Lunze and F.~Lamnabhi-Lagarrigue, Eds., \emph{{Handbook of Hybrid Systems
  Control}}.\hskip 1em plus 0.5em minus 0.4em\relax Cambridge University Press,
  2009.

\bibitem{bemporad99:_contr_system_integ_logic_dynam_const}
A.~Bemporad and M.~Morari, ``{Control of Systems Integrating Logics, Dynamics,
  and Constraints},'' \emph{Automatica}, vol.~35, no.~5, pp. 407 -- 427, 1999.

\bibitem{maciejowski01:_predic_contr_const}
J.~Maciejowski, \emph{Predictive Control with Constraints}.\hskip 1em plus
  0.5em minus 0.4em\relax Prentice Hall, June 2001.

\bibitem{rawlings09:_model_predic_contr}
J.~Rawlings and D.~Mayne, \emph{{Model Predictive Control: Theory and
  Design}}.\hskip 1em plus 0.5em minus 0.4em\relax Nob Hill Publishing, 2009.

\bibitem{bank82:_non_linear_param_optim}
B.~Bank, J.~Guddat, D.~Klatte, B.~Kummer, and K.~Tammer, \emph{{Non-Linear
  Parametric Optimization}}.\hskip 1em plus 0.5em minus 0.4em\relax Berlin:
  Akademie-Verlag, 1982.

\bibitem{pistikopoulos00:_optim_off_param_optim_tools}
E.~Pistikopoulos, V.~Dua, N.~Bozinis, A.~Bemporad, and M.~Morari, ``{On-line
  Optimization via Off-line Parametric Optimization Tools},'' in
  \emph{International Symposium on Process Systems Engineering}, Keystone, USA,
  Jul. 2000, pp. 183--188.

\bibitem{bemporad02:_model_predic_contr_based_linear_progr}
A.~Bemporad, F.~Borrelli, and M.~Morari, ``{Model Predictive Control Based on
  Linear Programming -- The Explicit Solution},'' \emph{IEEE Transactions on
  Automatic Control}, vol.~47, no.~12, pp. 1974 -- 1985, December 2002.

\bibitem{borrellimoraribook}
F.~Borrelli, A.~Bemporad, and M.~Morari, \emph{{Predictive control for linear
  and hybrid systems}}.\hskip 1em plus 0.5em minus 0.4em\relax In preparation,
  draft at http://www.mpc.berkeley.edu, 2013.

\bibitem{dua02:_multip_progr_approac_mixed_integ}
V.~Dua, N.~Bozinis, and E.~Pistikopoulos, ``{A Multiparametric Programming
  Approach for Mixed-Integer Quadratic Engineering Problems},'' \emph{Computers
  \& Chemical Engineering}, vol.~26, pp. 715--733, 2002.

\bibitem{baotic05:_optim_contr_piecew_affin_system}
M.~Baoti\'c, ``{Optimal Control of Piecewise Affine Systems -- A
  Multi-parametric Approach},'' Ph.D. dissertation, Automatic Control
  Laboratory, ETH Zurich, Mar. 2005.

\bibitem{axehill10_subop_automatica}
D.~Axehill, T.~Besselmann, D.~Raimondo, and M.~Morari, ``{A parametric branch
  and bound approach to suboptimal explicit hybrid MPC},'' \emph{Automatica},
  vol.~50, no.~1, pp. 240--246, Jan. 2014.

\bibitem{tondel03:_evaluat}
P.~Tondel, T.~Johansen, and A.~Bemporad, ``{Evaluation of piecewise affine
  control via binary search tree},'' \emph{Automatica}, vol.~39, no.~5, pp.
  945--950, 2003.

\bibitem{JonGri:2006:IFA_2452}
C.~Jones, P.~Grieder, and S.~Rakovic, ``{A Logarithmic-Time Solution to the
  Point Location Problem for Parametric Linear Programming},''
  \emph{Automatica}, vol.~42, no.~12, pp. 2215--2218, Dec. 2006.

\bibitem{WanJon:2007:IFA_3054}
Y.~Wang, C.~Jones, and J.~M. Maciejowski, ``{Efficient point location via
  subdivision walking with application to explicit MPC},'' in \emph{European
  Control Conference}, Kos, Greece, Jul. 2007.

\bibitem{kvasnica04:_multi_param_toolb_mpt}
M.~Kvasnica, P.~Grieder, M.~Baoti\'c, and M.~Morari, ``{Multi-Parametric
  Toolbox (MPT)},'' in \emph{Proceedings of the International Workshop on
  Hybrid Systems: Computation and Control}, Philadelphia, USA, Mar. 2004, pp.
  448--462.

\bibitem{ChrEtal:ecc:07}
F.~Christophersen, M.~Kvasnica, C.~Jones, and M.~Morari, ``{Efficient
  Evaluation of Piecewise Control Laws defined over a Large Number of
  Polyhedra},'' in \emph{Proc.~of the European Control Conference}, Kos,
  Greece, Jul. 2007.

\bibitem{FucAxe:2010:IFA_3656}
A.~Fuchs, D.~Axehill, and M.~Morari, ``{On the choice of the Linear Decision
  Functions for Point Location in Polytopic Data Sets - Application to Explicit
  MPC},'' in \emph{IEEE Conference on Decision and Control}, Atlanta, USA, Dec.
  2010, pp. 5283--5288.

\bibitem{besselmann10:_const_optim_contr}
T.~Besselmann, ``{Constrained Optimal Control -- Piecewise Affine and Linear
  Parameter-Varying Systems},'' Ph.D. dissertation, Automatic Control
  Laboratory, ETH Zurich, 2010.

\bibitem{Mato}
M.~Baoti\'c, ``{Polytopic Computations in Constrained Optimal Control },''
  \emph{Automatika}, vol.~50, no.~3, pp. 119 -- 134, 2009.

\bibitem{FucJon:2010:IFA_3625}
A.~Fuchs, C.~Jones, and M.~Morari, ``{Optimized Decision Trees for Point
  Location in Polytopic Data Sets - Application to Explicit MPC},'' in
  \emph{American Control Conference}, Baltimore, USA, Jun. 2010.

\bibitem{partree}
M.~M\"onnigmann and M.~Kastsian, ``Fast explicit {MPC} with multiway trees,''
  \emph{IFAC World Congress}, Sep. 2011.

\bibitem{dcdc}
T.~Geyer, G.~Papafotiou, and M.~Morari, ``{Hybrid Model Predictive Control of
  the Step-Down DC-DC Converter},'' \emph{IEEE Transactions on Control Systems
  Technology}, vol.~16, no.~6, pp. 1112--1124, Nov. 2008.

\end{thebibliography}
